\documentclass[a4paper,12pt]{amsart}

\usepackage[latin1]{inputenc}
\usepackage[T1]{fontenc}
\usepackage[english]{babel}

\usepackage{mathrsfs}
\usepackage{amscd}
\usepackage{amsfonts}
\usepackage{amsmath}
\usepackage{amssymb}
\usepackage{amstext}
\usepackage{amsthm}
\usepackage{amsbsy}

\usepackage{xspace}
\usepackage[all]{xy}
\usepackage{graphicx}
\usepackage{url}
\usepackage{latexsym}

\makeatletter
\newcommand*{\rom}[1]{\expandafter\@slowromancap\romannumeral #1@}
\makeatother

\theoremstyle{definition}

\newtheorem{fact}{fact}

\newtheorem{thm}[fact]{Theorem}
\newtheorem{lemma}[fact]{Lemma}
\newtheorem{prop}[fact]{Proposition}
\newtheorem{corollary}[fact]{Corollary}
\newtheorem{defini}[fact]{Definition}

\title{Optimal results on $ITRM$-recognizability}
\author{Merlin Carl}

\begin{document}

\maketitle

\begin{abstract}
Exploring further the properties of $ITRM$-recognizable reals started in \cite{Ca}, we provide a detailed analysis of recognizable reals and their distribution in G\"odels constructible universe $L$. 
In particular, we show that, for unresetting infinite time register machines, the recognizable reals coincide with the computable reals and that, for $ITRM$s, unrecognizables are generated
at every index $\gamma\geq\omega_{\omega}^{CK}$. We show that a real $r$ is recognizable iff it is $\Sigma_{1}$-definable over $L_{\omega_{\omega}^{CK,r}}$, that $r\in L_{\omega_{\omega}^{CK,r}}$ for every
recognizable real $r$ and that either all or no real generated over an index stage $L_{\gamma}$ are recognizable.
\end{abstract}

\section{Introduction}

Infinite Time Register Machines ($ITRM$'s) and weak (or unresetting) Infinite Time Register Machines ($wITRMs$) 
are a machine model for infinite computations introduced by Peter Koepke and Russell Miller in \cite{KoMi} and \cite{wITRM}, respectively. We will describe these models only shortly. Detailed descriptions of $ITRM$'s and
all of the results about these machines we will use here can be found in \cite{KoMi} and \cite{ITRM}.\\
An $ITRM$ resembles in most of its features a classical universal register machine ($URM$) from \cite{Cu}: It has finitely many registers $R_{1},...,R_{n}$ each of which can store one natural number. An $ITRM$-program consists
of finitely many lines, each of which contains one command. Commands are the increasing and decreasing of a register content by $1$, copying a register content to another register, reading out the $r_{i}$th bit of an oracle (where $r_{i}$
is the content of the $i$th register), jumping to a certain program line provided a certain register 
content is $0$, and stopping.\\
In contrast to $URM$'s, $ITRM$'s allow an arbirary ordinal as their running time. Accordingly, the definition of an $ITRM$-computation now has to take care of limit steps. At successor ordinals, we define the computation in the same way as for $URM$'s.
If $\lambda$ is a limit ordinal, we set the content $R_{i\lambda}$ of the $i$-th register $R_{i}$ at time $\lambda$ to $\liminf_{\iota<\lambda}R_{i\iota}$ iff this limit is finite, and to $0$ otherwise. Likewise, the active program line $Z_{\lambda}$
to be carried out in the $\lambda$th step is $\liminf_{\iota<\lambda}Z_{\iota}$, where the limit is always finite as the set of lines is finite and their indices are therefore bounded.\\

\begin{defini}
 $x\subseteq\omega$ is $ITRM$-computable in the oracle $y\subseteq\omega$ iff there exists an $ITRM$-program $P$ such that, for $i\in\omega$, $P$ with oracle $y$ stops whatever natural number $j$ is in its first register at the start of the computation and returns
$1$ iff $j\in x$ and otherwise returns $0$. A real $ITRM$-computable in the empty oracle is simply called $ITRM$-computable. 
\end{defini}

Apart from computability, which is a direct analogue of the corresponding finite concept, there is a different notion of how an $ITRM$ can 'handle' a real number, which has no interesting analogue in the finite. A classical $URM$ $R$ can only process a finite
part of each oracle, and hence, for each real $r$, there is an open neighbourhood $u$ of $r$ such that $R$ cannot distinguish the elements of $u$. The computing time of an $ITRM$, on the other hand, allows it to repeatedly consider each bit of a real number. 
Hence, it has a chance of identifying individual real numbers. Numbers for which this is possible are called 'recognizable'.

\begin{defini}
Let $r\subseteq\omega$. Then $r$ is recognizable iff there is an $ITRM$-program $P$ such that $P^{x}$ stops with output $1$ iff $x=r$ and otherwise stops with output $0$.
\end{defini}

Most of our notation is standard. $ZF^{-}$ is $ZF$ set theory without the power set axiom.
$\mathfrak{P}(x)$ will denote the power set of $x$. For an $ITRM$-program $P$, $P^{x}(i)\downarrow=j$ means that the program $P$ with oracle $x$ with initial input $i$
in its first register stops with output $j$ in register $1$. We take $R_1$ to be the generic register for input and output and will not care about such details in the further course of this paper. 
By $\omega_{\iota}^{CK}$, we denote the $\iota$-th admissible ordinal, where $\iota\in On$. When we consider admissible ordinals relative to a real $x$, we write $\omega_{\iota}^{CK,x}$. For $X\subseteq L_{\alpha}$, $\Sigma_{1}^{L_{\alpha}}\{X\}$ 
denotes the $\Sigma_{1}$-Skolem hull of $X$ in $L_{\alpha}$ and $\Sigma_{\omega}^{L_{\alpha}}\{X\}$ denotes the elementary hull of $X$ in $L_{\alpha}$. When $H$ is a $\Sigma_{1}$-substructure of some $L_{\alpha}$, 
then $\pi:H\equiv L_{\gamma}$ and $\pi:H\rightarrow_{coll}L_{\gamma}$ denote the Mostowski collapse of $H$ to $L_{\gamma}$ with isomorphism $\pi$. Throughout the paper, $p:\omega\times\omega\rightarrow\omega$ denotes the usual bijection
between $\omega\times\omega$ and $\omega$.

\section{Weak ITRMs}

\begin{thm}{\label{relwITRM}}
 Let $x,y\subseteq\omega$. Then $x$ is $wITRM$-computable in the oracle $y$ iff $x\in L_{\omega_{1}^{CK,y}}[y]$.
\end{thm}
\begin{proof}
A straightforward relativization of the proof of Theorem $1$ in \cite{wITRM}.
\end{proof}

\begin{defini}
Let us denote by $wRECOG$ the set of reals recognizable by a weak $ITRM$ and by $wCOMP$ the set of reals computable by a weak $ITRM$.
\end{defini}

(Recall that $wCOMP=HYP=\mathbb{R}\cap L_{\omega_{1}^{CK}}$.)

The following is a relativization of a theorem in \cite{Ba}:\\

\begin{lemma}{\label{relbar}}
Let $x\in\mathbb{R}$ and let $M\models KP$ be such that $\omega^{M}=\omega$ and $x\in M$. Then $\omega_{1}^{CK,x}$ is an initial segment of $On^{M}$.
\end{lemma}
\begin{proof}
\end{proof}

\begin{lemma}{\label{cyclecrit}}
Let $P$ be a $wITRM$-program, and let $x\subseteq\omega$. Then $P^{x}\uparrow$ iff there exist $\sigma<\tau<\omega_{1}^{CK,x}$ such that $Z(\tau)=Z(\sigma)$, $R_{i}(\theta)=R_{i}(\sigma)$ for all $i\in\omega$
and $R_{i}(\gamma)\geq R_{i}(\tau)$ for all $i\in\omega$, $\sigma<\theta<\omega_{1}^{CK,x}$. (Here, $Z(\gamma)$ and $R_{i}(\gamma)$ denote the active program line and the content of register $i$ at time $\gamma$.)
\end{lemma}
\begin{proof}
This is a straightforward relativization of a result of \cite{wITRM}.
\end{proof}

The following lemma allows us to quantify over countable $\omega$-models of $KP$ by quantifying over reals:\\

\begin{lemma}{\label{kpsigma1}}
 There is a $\Sigma_{1}^{1}$-statement $\phi(v)$ such that $\phi(x)$ holds only if $x$ codes an $\omega$-model of $KP$ and such that, for any countable $\omega$-model $M$ of $KP$, there is a code $c$ for $M$
such that $\phi(c)$ holds.
\end{lemma}
\begin{proof}
Every countable $\omega$-model $M$ of $KP$ can be coded by a real $c(M)$ in such a way that the $i\in\omega$ is represented by $2i$ in $c$ and $\omega$ is represented by $1$.
We can then consider a set $S$ of statements saying that a real $c$ codes a model of $KP$ together with $\{P_{k}|k\in\omega+1\}$, where $P_{k}$ is the statement $\forall{i<k}(p(2i,2k)\in c)\wedge\forall{j}\exists{i<k}(p(j,k)\in{c}\rightarrow j=2i)$ for 
$k\in\omega$ and $P_{\omega}$ is the statement $\forall{i}(p(2i,1)\in c)\wedge\forall{j}\exists{i}(p(j,1)\in c\rightarrow j=2i)$.
Then $\bigwedge{S}$ is a hyperarithmetic conjunction of arithmetic formulas in the predicate $c$. But such a conjunction is equivalent to a $\Sigma_{1}^{1}$-formula.
\end{proof}

\begin{thm}{\label{wsigma1}}
Let $x$ be recognizable by a $wITRM$. Then $\{x\}$ is a $\Sigma_{1}^{1}$-singleton.
\end{thm}
\begin{proof}
Let $P$ be a program that recognizes $x$ on a $wITRM$. Let $KP(z)$ be a $\Sigma_{1}^{1}$-formula (in the predicate $z$) stating that $z$ codes an $\omega$-model of $KP$ with $\omega$ represented by $1$
and every integer $i$ represented by $2i$ as constructed in Lemma \ref{kpsigma1}. 
Let $E(y,z)$ be a first-order formula (in the predicates $y$ and $z$) stating that the structure coded by $z$ contains $y$. (We can e.g. take $E(y,z)$ to be $\exists{k}\forall{i}(z(i)\leftrightarrow z(p(2i,k)))$.)
Furthermore, let $Acc_{P}(z,y)$ be a first-order formula (in the predicates $y$ and $z$) stating that $P^{y}\downarrow=1$ in the structure coded by $z$. (Check possibility!)
Finally, let $NC_{P}(y)$ be a first-order formula (in the predicate $y$) stating that in the computation $P^{y}$, there are no two states $s_{\iota_{1}},s_{\iota_{2}}$ with $s_{\iota_{1}}<s_{\iota_{2}}$
such that $s_{\iota_{1}}=s_{\iota_{2}}$ and, between $\iota_1$ and $\iota_2$, no register content is every between the content in $s_{\iota_{1}}$ and no program line with index smaller than that of the 
active line in $s_{\iota_{1}}$ turns up. This is possible in $KP$ models containing $x$ since, by Lemma \ref{relbar} above, $\omega_{1}^{CK,x}$ is an inital segment of the well-founded part of
each such model and, by Lemma \ref{cyclecrit}, the computation either cycles before $\omega_{1}^{CK,x}$ or stops.
Now, take $\phi(a)$ to be $\exists{z}(KP(z)\wedge E(a,z)\wedge Acc_{P}(z,a)\wedge NC_{P}(a))$. This is a $\Sigma_{1}^{1}$-formula. We claim that $x$ is the only solution to $\phi(a)$:
To see this, first note that $x$ clearly is a solution, since $\omega_{1}^{CK,x}$ is an initial segment of every $KP$-model containing $x$ by Lemma \ref{relbar}.\\ 
On the other hand, assume that $b\neq x$. In this case, as $P$ recognizes $x$, we have $P^{b}\downarrow=0$ in the real world, and hence, by absoluteness of $wITRM$-(oracle)-computations for $KP$-models containing
the relevant oracles, also inside $L_{\omega_{1}^{CK,b}}[b]$. Now $L_{\omega_{1}^{CK,b}}[b]$ is certainly a countable $KP$-model containing $b$, hence a counterexample to $\phi(b)$, so $\phi(b)$ is false.
\end{proof}

\begin{corollary}
If a real $x$ is $wITRM$-recognizable, then it is $wITRM$-computable. Hence, there are no lost melodies for weak $ITRM$s.
\end{corollary}
\begin{proof}
By Kreisel's basis theorem (see \cite{Sa}, p. $75$), if $a\notin HYP$ and $B\neq\emptyset$ is $\Sigma_{1}^{1}$, then $B$ contains some element $b$ such that $a\not{\leq}_{h}b$. 
Now suppose that $x$ is $wITRM$-recognizable. By Theorem \ref{wsigma1}, $\{x\}$ is $\Sigma_{1}^{1}$ and certainly non-empty. If $x$ was not hyperarithmetical, then, by Kreisel's theorem, $\{x\}$ would
contain some $b$ such that $x\neq{\leq}_{h}b$. But the only element of $\{x\}$ is $x$, so $x\notin HYP$ implies $x\neq{\leq}_{h}x$, which is absurd. Hence $x\in HYP$. So $x$ is $wITRM$-computable.
\end{proof}



\section{ITRMs}


We summarize here some results on $ITRM$s which are relevant for our further development.

\begin{thm}{\label{hp}}
 Bounded halting problem, solvable uniformly in the oracle.\\
 Halting times of programs in oracle $x$ below $\omega_{\omega}^{CK,x}$. With $n$ registers below $\omega_{n+1}^{CK,x}$
\end{thm}
\begin{proof}
 See \cite{ITRM}. 
\end{proof}

\begin{defini}
 An ordinal $\alpha$ is called $\Sigma_{1}$-fixed iff there exists a $\Sigma_{1}$-statement $\phi$ such that $\alpha$ is minimal with the
property that $L_{\alpha}\models\phi$. Let $\sigma$ denote the supremum of the $\Sigma_{1}$-fixed ordinals.
\end{defini}

\begin{thm}
Denote by $RECOG$ the set of recognizable reals. Then $RECOG\subseteq L_{\sigma}$. Furthermore, for each $\gamma<\sigma$, there exists $\alpha<\sigma$ such that $[\alpha,\alpha+\gamma]$ contains unboundedly
many indices, but $(L_{\alpha+\gamma}-L_{\alpha})\cap RECOG=\emptyset$.
\end{thm}
\begin{proof}
 See \cite{Ca}.
\end{proof}

\begin{thm}{\label{relITRM}}
 Let $x,y\subseteq\omega$. Then $x$ is $ITRM$-computable in the oracle $y$ iff $x\in L_{\omega_{\omega}^{CK,y}}[y]$.
\end{thm}
\begin{proof}
This is a straightforward relativization of the main result of \cite{KoeMi}.
\end{proof}

\begin{lemma}
 Let $A\neq\emptyset$ be an $ITRM$-decidable set of reals such that $a\in L_{\omega_{\omega}^{CK,a}}$ for all $a\in A$. Then the $<_{L}$-minimal element of $A$ is recognizable.
\end{lemma}
\begin{proof}
Let $a$ be the $<_{L}$-minimal element of such an $A$. By Theorem \ref{relITRM}, there is an $ITRM$-program $P$ such that $P^{a}$ computes a code for some $L$-level $L_{\alpha}$ containing $a$.
Let $Q_{A}$ be an $ITRM$-program deciding $A$. Now $a$ can be recognized as follows: Given some $x\subseteq\omega$ in the oracle, first check whether $P^{x}(i)\downarrow$ for all $i\in\omega$, using
a halting problem solver for $P$ which exists by Theorem \ref{hp}. If not, then $x\neq a$. Otherwise, test whether $P^{x}$ computes a code $c$ for an $L$-level containing $x$. This can be done
using the techniques for evaluating the truth predicate in coded structures provided in the last section of \cite{ITRM}. If not, then $x\neq a$. Otherwise, test whether $Q_{A}(x)\downarrow=1$.
If not, then $x\notin A$, so $x\neq a$. Otherwise, use $c$ to search through all reals below $x$ in $<_{L}$ for a real $z<_{L}x$ such that $Q_{A}(z)\downarrow=1$. If such a real is found, then
$x\neq a$. Otherwise, $x=a$.
\end{proof}

\subsection{Unrecognizables Everywhere}

\begin{defini}
 An ordinal $\gamma$ is called an index iff $L_{\gamma+1}-L_{\gamma}$ contains a subset of $\omega$.
\end{defini}

\begin{thm}{\label{largeCK}}
Let $x\in RECOG$. Then $x\in L_{\omega_{\omega}^{CK,x}}$. Consequently, if $x$ is recognizable, but not $ITRM$-computable, we have $\omega_{\omega}^{CK,x}>\omega_{\omega}^{CK}$.
\end{thm}
\begin{proof}: Let $P$ be a program that recognizes $x$. Then $L_{\omega_{\omega}^{CK,x}}[x]\models\exists{y}P^{y}\downarrow=1$. Now $\phi=\exists{y}P^{y}\downarrow=1$ is a (set theoretical) $\Sigma_{1}$-statement, basically stating that
there are a real $y$ and a set $c$ such that $c$ codes the $P$-computation in the oracle $y$ and ends with $1$. By Jensen-Karp (see above), this is absolute between $V_{\alpha}$ and $L_{\alpha}$ whenever $\alpha$ is a limit
of admissibles. Now $\omega_{\omega}^{CK,x}=\sup\{\omega_{i}^{CK,x}\}$ certainly is a limit of admissibles, so $\phi$ is absolute between $V_{\alpha}$ and $L_{\omega_{\omega}^{CK,x}}$. Also, as $\phi$ is $\Sigma_{1}$, it is certainly
upwards absolute. Hence $L_{\omega_{\omega}^{CK,x}}[x]\models\phi\implies V_{\omega_{\omega}^{CK,x}}\models\phi\implies L_{\omega_{\omega}^{CK,x}}\implies L_{\omega_{\omega}^{CK,x}}[x]$, so $\phi$ is absolute between 
$L_{\omega_{\omega}^{CK,x}}[x]$ and $L_{\omega_{\omega}^{CK,x}}$. As $\phi$ holds in $L_{\omega_{\omega}^{CK,x}}[x]$, it follows that $\phi$ holds in $L_{\omega_{\omega}^{CK,x}}$. So $L_{\omega_{\omega}^{CK,x}}$ contains a real $r$
such that $P^{r}\downarrow=1$. By absoluteness of computations, $P^{r}\downarrow=1$ also holds in $V$. So $P^{r}\downarrow=1$. As $P$ recognizes $x$, it follows that $x=r$. Hence $x\in L_{\omega_{\omega}^{CK,x}}$.\\
Now let $x$ be recognizable, but not computable. As $x$ is not computable, we have $x\notin L_{\omega_{\omega}^{CK}}$. By the first part of the claim, $x\in L_{\omega_{\omega}^{CK,x}}$. Hence $\omega_{\omega}^{CK,x}>\omega_{\omega}^{CK}$.
\end{proof}

This immediately leads to the following dichotomy:

\begin{corollary}{\label{dichotomy}}
 If $x\subseteq\omega$ is such that $\omega_{\omega}^{CK,x}=\omega_{\omega}^{CK}$, then either $x$ is $ITRM$-computable or $x$ is not $ITRM$-recognizable.
\end{corollary}
\begin{proof}
If $x$ is $ITRM$-computable, then $x$ is clearly $ITRM$-recognizable. If $x$ is not $ITRM$-computable, then $x\notin L_{\omega_{\omega}^{CK}}$, hence $x\notin L_{\omega_{\omega}^{CK,x}}$ if $\omega_{\omega}^{CK,x}=\omega_{\omega}^{CK}$.
By the last theorem then, $x$ is not recognizable.
\end{proof}

\begin{lemma}{\label{mat}}
 Let $\alpha$ be admissible, $(P,\leq)\in L_{\alpha}$ be a notion of forcing and $G$ be a filter on $P$ such that $P\cap D\neq\emptyset$ for every dense subset $D$ of $P$ such that $D\in L_{\alpha+1}$. Then 
$L_{\alpha}[G]$ is admissible.
\end{lemma}
\begin{proof}
 This follows from Theorem $10.1$ of \ref{mat}, since unions of $\Sigma_{1}(L_{\alpha})$ and $\Pi_{1}(L_{\alpha})$-definable subsets of $P$ are clearly elements of $L_{\alpha+1}$.
\end{proof}

\begin{corollary}{\label{matcorr}}
Let $\gamma\geq\omega_{\omega}^{CK}$, let $(P,\leq_{P})$ be the notion of forcing for adding a Cohen real (i.e. $P$ consists of the finite partial functions from $\omega$ to $2$ and $x\leq_{P}y$ iff $y\subseteq x$) and let $G$ be a filter
on $(P,\leq)$ which intersects every dense $D\subseteq P$ such that $D\in L_{\gamma}$. Then $L_{\omega_{i}^{CK}}[G]$ is admissible for every $i\in\omega$.
\end{corollary}
\begin{proof}
 This is immediate from Lemma \ref{mat} as $L_{\gamma}\supseteq L_{\omega_{\omega}^{CK}}\supseteq L_{\omega_{i}^{CK}+1}$ for all $i\in\omega$.
\end{proof}

The following will be used to show that, for each index $\gamma\geq\omega_{\omega}^{CK}$, $L_{\gamma+1}-L_{\gamma}$ contains an unrecognizable real. Recall that $\gamma\in On$ is said to be an index iff $(L_{\gamma+1}-L_{\gamma})\cap^{2}\omega\neq\emptyset$.

\begin{thm}{\label{lowadmissibles}}
 Let $\gamma\geq\omega_{\omega}^{CK}$ be an index. Then there exists $x\in L_{\gamma+1}-L_{\gamma}$ with $\omega_{i}^{CK,x}=\omega_{i}^{CK}$ for all $i\in\omega$. In particular, this implies that $\omega_{\omega}^{CK,x}=\omega_{\omega}^{CK}$.
\end{thm}
\begin{proof} Let $P$ be the notion of forcing for adding a Cohen real (see above). Let $G$ be an $L_{\gamma}$-generic filter on $P$ (i.e. $G$ intersects every dense subset of $P$, inside $L_{\gamma}$ liegt). 
By Corollary \ref{matcorr}, $L_{\omega_{i}^{CK}}[G]$ is then admissible for all $i\in\omega$.\\
Now let $x:=\bigcup G\in^{2}\omega$. We show that $L_{\omega_{i}^{CK}}[G]=L_{\omega_{i}^{CK}}[x]$:
As $x=\bigcup{G}\in L_{\omega_{i}^{CK}}[G]$, we have $L_{\omega_{i}^{CK}}[x]\subseteq L_{\omega_{i}^{CK}}[G]$ and $G\in L_{\omega_{i}^{CK}}[x]$ (since $G$ is definable from $x$),
hence also $L_{\omega_{i}^{CK}}[G]\subseteq L_{\omega_{i}^{CK}}[x]$. More generally, if $\alpha$ is additively indecomposable (which certainly holds for $\alpha$ admissible) and if we have $x\in L_{\alpha}[y]$ and $y\in L_{\alpha}[x]$, 
then $L_{\alpha}[x]=L_{\alpha}[y]$. To see this, let $z\in L_{\beta}[x]$ ($\beta<\alpha$) and $x\in L_{\gamma}[y]$ ($\gamma<\alpha$). Then $z\in L_{\beta}[x]\in  L_{\gamma+\beta+1}[y]\subseteq L_{\alpha}[y]$, hence
$L_{\alpha}[x]\subseteq L_{\alpha}[y]$. $L_{\alpha}[y]\subseteq L_{\alpha}[x]$ now follows by symmetry. \\
Now it follows that $L_{\omega_{1}^{CK}}[x]\models KP$, i.e. $\omega_{1}^{CK}$ is $x$-admissible, so that $\omega_{1}^{CK}\geq\omega_{1}^{CK,x}\geq\omega_{1}^{CK}$. Consequently, we get $\omega_{1}^{CK,x}=\omega_{1}^{CK}$.
Now assume inductively that $\omega_{i}^{CK}=\omega_{i}^{CK,x}$ for some $i\in\omega$. It then follows that $L_{\omega_{i+1}^{CK}}[x]\models KP$, hence $\omega_{i+1}^{CK}$ is $x$-admissible and $>\omega_{i}^{CK,x}=\omega_{i}^{CK}$.
But then $\omega_{i+1}^{CK}\geq\omega_{i+1}^{CK,x}\geq\omega_{i+1}^{CK}$, so $\omega_{i+1}^{CK,x}=\omega_{i+1}^{CK}$. This now gives us $\omega_{i}^{CK,x}=\omega_{i}^{CK}$ for all $i\in\omega$, so that $\omega_{\omega}^{CK,x}=\omega_{\omega}^{CK}$.\\
Next, we demonstrate that $G$ - and hence $x=\bigcup{G}$ - are definable over $L_{\gamma}$ and hence elements of $L_{\gamma+1}$. This can be seen as follows: 
As $\gamma$ is an index, there is $f:\omega\rightarrow L_{\gamma}$ surjective such that $f\in L_{\gamma+1}$ and hence definable over $L_{\gamma}$ 
Now define $g:(^{<\omega}2,\omega)\rightarrow^{<\omega}2$ thus: Let $g(\vec{s},i)$ be the lexically minimal element of $f(i)$, of which $\vec{s}$ is a subsequence if $f(i)\subseteq^{<\omega}2$ is a dense subset of $P$
otherwise let $g(\vec{s},i)=\vec{s}$. Now, define recursively: $h(0)=\emptyset$, $h(i+1)=g(h(i),i+1)$. This recursion can be carried out definably over $L_{\gamma}$ as follows: Set (for $i\in\omega$) $h(i)=x$ iff\\
$(i=0\wedge x=\emptyset)\vee(i\geq 1\wedge\exists(\vec{s_{0}},...,\vec{s_{i-1}})[\forall{j\in i}((j=0\wedge\vec{s_{0}}=\emptyset)\vee(\vec{s_{j}}=g(\vec{s_{j-1}},j)))\wedge x=g(\vec{s_{i-1}},i)])$.
This is definable over $L_{\gamma}$, as $g$ is definable over $L_{\gamma}$ and all finite sequences of elements of $P$ are contained in $L_{\omega}$, and hence certainly in $L_{\omega_{\omega}^{CK}}$.\\
Finally, we show that $x\notin L_{\gamma}$: Roughly, this follows immediately from the fact that $G$ is definable from $x$ and that $G$ is generic over $L_{\gamma}$ as in the case of Cohen-forcing for $ZFC$ models. 
More precisely, let $z\in\mathbb{R}\cap L_{\gamma}$. Also, let $\beta<\gamma$ be minimal such that $z\in L_{\beta+1}-L_{\beta}$.
Then $D_{z}:=\{b\in^{<\omega}2|\exists{i\in\omega}b(i)\neq z(i)\}$ is dense in $P$ and definable over $L_{\beta}$, hence an element of $L_{\gamma}$. Consequently, every $D_{z}$ has non-emptys intersection with every $L_{\gamma}$-generic filter $G$,
so that $\bigcup{G}\neq z$. As this holds for all $z\in L_{\gamma}$, we get $\bigcup{G}\notin L_{\gamma}$.
\end{proof}

We can now show that new unrecognizables appear wherever possible, i.e. are generated at every index stage:

\begin{thm}
Let $\gamma\geq\omega_{\omega}^{CK}$ be an index. Then $L_{\gamma+1}-L_{\gamma}$ contains an unrecognizable real.
\end{thm}
\begin{proof}
By Theorem \ref{lowadmissibles}, $L_{\gamma+1}-L_{\gamma}$ contains a real $x$ such that $\omega_{\omega}^{CK,x}=\omega_{\omega}^{Ck}$. By Corollary \ref{dichotomy}, $x$ is not recognizable.
\end{proof}

\section{The halting number is recognizable}

We obtain a very natural lost melody by showing that the halting number for $ITRM$s is in fact recognizable. Fix a canonical well-ordering $(P_{i}|i\in\omega\}$ of the $ITRM$-programs in order type $\omega$ by e.g. sorting
the programs lexicographically. This real $h$ is natural insofar its definition is purely internal to $ITRM$s (e.g. not in any way related to $L$) and it is arguably the first non-computable real coming to mind.\\

We start by showing that, given $h$, there is a universal $ITRM$:\\

\begin{lemma}
 There is an $ITRM$-program $P$ such that, for every $(i,j)\in\omega^2$, we have $P^{h}(p(i,j))\downarrow=k+1$ if $P_{i}(j)\downarrow=k$ and $P^{h}(p(i,j))\downarrow=0$ if $P_{i}(j)\uparrow$. That is, $P$ can compute the function
computed by $P_{i}$ given $i$.
\end{lemma}
\begin{proof}
 $P$ works as follows: Given $i$ and $j$, first use $h$ to check whether $P_{i}(j)\downarrow$. If $P_{i}(j)\uparrow$, $P$ returns $0$. Otherwise, we carry out the following procedure for each $k\in\omega$:
Compute (which can be done with a standard register machine, in fact) an index $l$ such that $P_{l}\downarrow$ iff $P_{i}(j)\downarrow=k$. $P_{l}$ will use a halting problem solver for $P_{i}$ (which can be easily
obtained from $P_{i}$), i.e. a sub-program $Q$ such that $Q(j)\downarrow=1$ iff $P_{i}(j)\downarrow$ and $Q(j)\downarrow=0$, otherwise. If it turns out that $Q(j)=0$, then $P_{l}$ enters an infinite loop.
Otherwise, we wait until $P_{i}(j)$ has stopped and check whether the outcome is $k$. If it is, we stop, otherwise we enter an infinite loop. (Note that $P$ is not required to do all this; it is only required that $P$
can compute a code for a program that does this, which is in fact easy).\\
Using $l$ and $h$, we can easily check whether $P_{i}(j)\downarrow=k$. If so, we return $k+1$. Otherwise, we continue with $k+1$.\\
As $P_{i}(j)\downarrow$ is already clear at this point, this has to lead to the value of $P_{i}(j)$ after finitely many steps.
\end{proof}

The next step is that, using $h$, a code for $L_{\omega_{i}^{CK}}$ can be computed uniformly in $i$.

\begin{corollary}
There is an $ITRM$-program $Q$ such that, for every $i\in\omega$, $Q^{h}(i)$ computes a code for $L_{\omega_{i}^{CK}}$. (I.e.: $Q^{h}(n)$ halts for every $n\in\omega$ and $\{j\in\omega|Q^{h}(p(i,j)\downarrow=1\}$ will be a code
for $L_{\omega_{i}^{CK}}$.)
\end{corollary}
\begin{proof}
First note that codes for $L_{\omega_{i}^{CK}}$ are uniformly recognizable in $i$, i.e. there is a program $R$ such that, for every $i\in\omega$, $x\subseteq\omega$, $R^{x}(i)\downarrow=1$ iff
$x$ codes $L_{\omega_{i}^{CK}}$ and otherwise $R^{x}(i)\downarrow=0$. This can be obtained using the well-foundedness checker combined with the first-order checker described in [LoMe] for
$V=L+KP+$'There are exactly $i-1$ admissible ordinals'.\\
Using $h$, we can now run through $\omega$, first testing whether $P_{k}(j)$ will halt for each $j\in\omega$ and then, using $P$ from the last lemma, whether $P_{k}$ will compute a code for $L_{\omega_{i}^{CK}}$. (We can evaluate
$P_{k}(j)$ for every $j$ using $P$ from the last lemma and then use $R$ to recognize whether the computed number is a code.)\\
As $L_{\omega_{i}^{CK}}$ has $ITRM$-computable codes, the minimal index $l$ such that $P_{l}$ computes a code for $L_{\omega_{i}^{CK}}$ will eventually be found in this way.\\
After that, we can, again using $P$ from the last lemma, evaluate $l$ to compute the desired code.
\end{proof}

These bits can now be put together to form a code for $L_{\omega_{\omega}^{CK}}$. This code will be a bit different from the codes considered so far, as we allow one element of the coded structure to be represented by arbitrary many
elements of $\omega$.\\

\begin{defini}
 Let $(X,\in)$ be a transitive $\in$-structure. Furthermore, let $f:\omega\rightarrow X$ be surjective. Then $\{p(i,j)|f(i)\in f(j)\}$ is called an odd code for $(X,\in)$.
\end{defini}

Odd codes can be evaluated in the same way that the codes we used so far could. The possibility of elements appearing repeatedly hinders none of those methods. It is helpful, however, to note that the equality is computable:\\

\begin{prop}
There is an $ITRM$-program $\bar{T}$ such that, for every odd code $x$ for a transitive $\in$-structure $(X,\in)$ (with associated function $f:\omega\rightarrow X$) and all $i,j\in\omega$, $\bar{T}^{x}(p(i,j))\downarrow=1$ iff $f(i)=f(j)$ and
$T^{x}\downarrow=0$, otherwise.\\
Furthermore, there is an $ITRM$-program $T$ such that, for every two odd codes $x$ and $y$ for transitive $\in$-structures $(X,\in)$ and $(Y,\in)$ (with associated functions $f_1$ and $f_2$), $T^{x\oplus y}(p(i,j))\downarrow=1$ iff
$f_{1}(i)=f_{2}(j)$ and $T^{x\oplus y}(p(i,j))\downarrow=0$, otherwise.
\end{prop}
\begin{proof}
 An easy application of the techniques developed in \cite{ITRM}.
\end{proof}

\begin{lemma}
There is an $ITRM$-program $S$ such that $S^{h}$ computes an odd code for $L_{\omega_{\omega}^{CK}}$.
\end{lemma}
\begin{proof}
Basically we reserve $\omega$ bits for coding $L_{\omega_{i}^{CK}}$; in one portion (the $i$-th portion), we use $Q^{h}$ to compute a code for $L_{\omega_{i}^{CK}}$. Then we use $T$ from the last proposition to relate the portions.
\end{proof}

\begin{thm}
 Let $h:=\{i\in\omega|P_{i}\downarrow\}$ be the set of indices of halting $ITRM$-programs in the canonical enumeration of programs. Then $h\in RECOG$.
\end{thm}
\begin{proof}
Let $x$ be the real in the oracle. Check whether $S^{x}$ computes an odd code $c$ for $L_{\omega_{\omega}^{CK}}$. If not, return $0$. 
Checking whether certain programs halt amounts to checking whether certain first-order statements hold in $L_{\omega_{\omega}^{CK}}$, which can be done using $c$. Then compare the results with $x$.
This identifies $h$.
\end{proof}

This idea should generalize to yield that also $h^{h}:=\{i\in\omega|P_{i}^{h}\downarrow\}$ is recognizable. More generally, let $h_0 :=h$, $h_{i+1}:=\{j\in\omega|P_{j}^{h_{i}}\downarrow\}$, then
this should show that $h_{i}\in RECOG$ for every $i\in\omega$. Uniformising this might even lead higher up, e.g. $h_{\omega}:=\{p(i,j)|i\in\omega\wedge j\in\omega\wedge i\in h_{j}\}\in RECOG$.
What is the first $\alpha$ such that $h_{\alpha}\notin RECOG$?\\

\textbf{Questions}: Is there a non-computable real $x$ such that $h$ does not reduce to $x$? If so, is there such a real which is not recognizable? Generally: How do computability degrees relate to recognizability?

\subsection{Optimal results on the distribution of recognizables}

We saw above (via Jensen-Karp) that $x\in RECOG$ implies that $x\in L_{\omega_{\omega}^{CK,x}}$. Reals without this property are hence ruled out, we concentrate on those that have it.

\begin{defini}
 $x\subseteq\omega$ is potentially recognizable iff $x\in L_{\omega_{\omega}^{CK,x}}$. We denote the set of potentially recognizable reals by $PRECOG$.
\end{defini}

\begin{thm}{\label{allornothing}}
 Let $\gamma$ be an index. Then either all potentially recognizable elements of $L_{\gamma+1}-L_{\gamma}$ are recognizable or none is.
\end{thm}
\begin{proof}
 (Sketch) Suppose $a\in (L_{\gamma+1}-L_{\gamma})\cap RECOG$ and $x\in (L_{\gamma+1}-L_{\gamma})\cap PRECOG$. We want to show that $x\in RECOG$. Pick a program $Q$ that recognizes $a$. As $x\in PRECOG$, there is $i\in\omega$ such that $x\in L_{\omega_{i}^{CK,x}}$.
In particular, we have $L_{\gamma+1}\in L_{\omega_{i}^{CK,x}}$. Hence $c=cc(L_{\gamma+1})$, the $<_{L}$-minimal real code for $L_{\gamma+1}$ is computable from $x$. Let $P$ be a program that computes $c$ from $x$.\\
To identify whether $y=x$ (with $y$ in the oracle), we first use the halting problem solver for $P$ to check whether $P^{y}(i)\downarrow$ for all $i\in\omega$. If not, then $y\neq x$. If yes, we check whether
$P^{y}$ computes a code $d$ for an $L$-level containing $y$. If not, then $y\neq x$. If yes, we use the technique from the proof of the Lost Melody Theorem to check whether $d$ is $<_{L}$-minimal with that property.
If not, then $y\neq x$. If yes, we check whether the structure coded by $d$ contains a real $r$ such that $Q^{r}\downarrow=1$. This can be done using the halting problem solver for $Q$. If there is no such $r$, then 
$y\neq x$. If there is, we check whether the structure coded by $d$ contains an $L$-level that also contains $r$ (this checks the minimality of $\gamma$). If not, then $y\neq x$, otherwise, $y=x$. So this procedure recognizes $x$, hence $x\in RECOG$.
\end{proof}

\begin{thm}{\label{char}}
Let $x\in PRECOG$. Then $x\in RECOG$ iff there exists a $\Sigma_{1}$-formula $\phi$ of set theory without parameters such that $x$ is the unique witness for $\phi(v)$ in $L_{\omega_{\omega}^{CK,x}}$.
\end{thm}
\begin{proof}
 (Sketch) If $x\in RECOG$ and $P$ recognizes $x$, then $P^{x}\downarrow=1$ is $\Sigma_{1}$-expressable over $L_{\omega_{\omega}^{CK,x}}$ (for $x\in PRECOG$).\\
On the other hand, if $x$ is definable as above, then let $L_{\gamma}$ be the first $L$-level containing $x$ such that $L_{\gamma}\models\phi(x)$. Then $\gamma<\omega_{\omega}^{CK,x}$,
so $c:=cc(L_{\gamma})$ can be computed from $x$, say by program $P$. Using $c$, we can check whether $L_{\gamma}\models\phi(x)$ holds.\\
Checking whether $y=x$ then works as follows: Check whether $P^{y}$ computes a minimal code for an $L$-level containing $y$, then check whether $\phi(y)$ holds in that $L$-level and then 
whether it fails in all earlier $L$-levels. If all of this holds, then $y=x$ (since $\Sigma_{1}$ is preserved upwards).
\end{proof}

\begin{corollary}
 For all $x\subseteq\omega$, $x$ is recognizable iff $x\in L_{\omega_{\omega}^{CK,x}}$ and $L_{\omega_{\omega}^{CK,x}}\models RECOG(x)$. In particular, if $L_{\alpha}\models ZF^{-}$ and $x\in\mathfrak{P}(\omega)\cap L_{\alpha}$,
then $x\in RECOG$ holds iff $L_{\alpha}\models RECOG(x)$.
\end{corollary}
\begin{proof}
 Suppose first that $x\in RECOG$, and let $P$ be a program that recognizes $x$. Then $x\in L_{\omega_{\omega}^{CK,x}}$ by Theorem \ref{largeCK}. 
By \cite{Ch}, if $z\in L_{\gamma}$ and $\gamma^{+}$ is the smallest admissible ordinal greater than $\gamma$, then $\omega_{1}^{CK,z}\leq\gamma^{+}$.
Inductively, we get that $\omega_{i}^{CK,z}\leq\gamma^{+i}$, where $\gamma^{+i}$ is the $i$th admissible ordinal above $\gamma$. Inductively, it follows that $\omega_{\omega}^{CK,z}\leq\omega_{\omega}^{CK,x}$ for all $z\in L_{\omega_{\omega}^{CK,x}}$
when $x$ is such that $x\in L_{\omega_{\omega}^{CK,x}}$. This implies that $P^{z}$ stops after at most $\omega_{\omega}^{CK,x}$ many steps for all $z\in L_{\omega_{\omega}^{CK,x}}$ and hence that $P^{z}$ can be carried out inside
$L_{\omega_{\omega}^{CK,x}}$ for all $z\in L_{\omega_{\omega}^{CK,x}}$. Hence, since $P$ recognizes $x$, we have $L_{\omega_{\omega}^{CK,x}}\models P^{z}\downarrow=0$ for all $z\neq x$ and furthermore $L_{\omega_{\omega}^{CK,x}}\models P^{x}\downarrow=1$.
Hence $L_{\omega_{\omega}^{CK,x}}\models RECOG(x)$.\\
On the other hand, assume that $x\in L_{\omega_{\omega}^{CK,x}}$ and that $L_{\omega_{\omega}^{CK,x}}\models RECOG(x)$. Hence $P^{x}\downarrow=1$ and $P^{z}\downarrow=0$ for all $z<_{L}x$. Now let $Q$ be a program such that $Q^{x}$ computes
the $<_{L}$-minimal code of the first $L$-level containing $x$. Then $x$ can be recognized as follows: Given some real $r$ in the oracle, first check, using a halting problem solver for $P$, whether $P^{r}\downarrow=1$. If not, then $r\neq x$.
Otherwise check - using a halting problem solver for $Q$ - whether $Q^{r}(i)\downarrow$ for all $i\in\omega$.
If not, then $r\neq x$. If yes, check whether $Q^{r}$ codes a minimal $L$-level containing $r$. If not, then $r\neq x$. If yes, check whether $Q^{r}$ is $<_{L}$-minimal with this property, using the usual strategy. If not, then $r\neq x$.
Otherwise, use $Q^{r}$ (and the halting problem solver for $P$) to check whether there is any real $y<_{L}x$ such that $P^{y}\downarrow=1$. If that is the case, then $r\neq x$. If it isn't, then $r$ is $<_{L}$-minimal
with $P^{r}\downarrow=1$ and hence $r=x$.
\end{proof}

%

\begin{defini}
 $\alpha\in\omega_{1}$ is admissibly $\Sigma_{1}$-describable iff there exists a $\Sigma_{1}$-formula $\phi$ of set theory without parameters such that $cc(\alpha)$ is the unique witness for $\phi(v)$ in $L_{\omega_{\omega}^{CK,cc(\alpha)}}$.
If $\alpha$ is not admissibly $\Sigma_{1}$-describable, we call it admissibly $\Sigma_{1}$-indescribable.
\end{defini}

\begin{defini}
 A strong substantial gap is an ordinal interval $[\alpha,\beta]$ such that every $\gamma\in[\alpha,\beta]$ is an index and such that $L_{\beta+1}-L_{\alpha}$ contains no recognizables.
 A weak substantial gap is an ordinal interval $[\alpha,\beta]$ such that $\alpha$ is an index, the set of indices in that interval is unbounded in $\beta$ and such that $L_{\beta+1}-L_{\alpha}$ contains no recognizables
\end{defini}

We can now show that gaps in the recognizables are never short:

\begin{thm}
There are no strong substantial gaps of finite length. Furthermore, strong gaps always start with limit ordinals.
\end{thm}
\begin{proof}
Assume for a contradiction that there is a strong substantial gap of length $i$, where $i\in\omega$. Let $\alpha\in On$ be minimal such that $[\alpha,\alpha+i]$ is a strong
substantial gap. It is easy to see that $cc(L_{\alpha+i})$ is recognizable by the usual arguments: Given $x$, check whether $x$ codes an $L$-level at which a strong substantial
gap of length $i$ ends. This can be done by the routines for evaluating truth predicates described in \cite{ITRM}. The minimality of $x$ can then also be checked by the techniques described
there. By the results on the computational strength of $ITRM$s, one readily obtains that from the $<_{L}$-minimal code $c$ of $L_{\alpha}$ which is not an element of $L_{\alpha}$, we can compute $cc(L_{\alpha+i})$,
say by program $P$.
But this allows us to recognize $c$: Given the oracle $x$, first check (using a halting problem solver for $P$) whether $P^{x}$ computes $cc(L_{\alpha+i})$ - which is possible as $cc(L_{\alpha+i})$ is recognizable.
Now, in $cc(L_{\alpha+i})$, $c$ is represented by some integer $j$. It hence only remains to see whether $x$ is the number represented by $j$ in $cc(L_{\alpha+i})$, which is also easy to do.\\
This implies that $c$ is recognizable. But, by definition, $c\in L_{\alpha+1}-L_{\alpha}$. Hence $(L_{\alpha+i}-L_{\alpha})\cap RECOG\neq\emptyset$, which contradicts the assumption that $\alpha$ starts a gap.\\
To see that, if $\alpha$ starts a strong substantial gap, $\alpha$ has to be a limit ordinal, we proceed as follows: Assume for a contradiction that $\alpha$ starts a strong substantial gap and $\alpha=\beta+1$.
Since $\alpha$ starts the gap, $L_{\alpha}-L_{\beta}$ contains a recognizable real $r$. We argue that $cc(L_{\alpha})\in L_{\alpha+1}-L_{\alpha}$ is recognizable, which contradicts the assumption that
$\alpha$ starts a gap. A procedure for describing $cc(L_{\alpha})$ works as follows: Given $x$, simply check whether $x$ is the $<_{L}$-minimal code of a minimal $L$-level containing $r$. This is possible since
$r$ is recognizable.
\end{proof}

The same reasoning in fact supports much stronger conclusions:

\begin{thm}
If $\alpha$ starts a weak substantial gap $[\alpha,\beta]$, then $\beta\geq\omega_{\omega}^{CK,cc(\alpha)}$.
\end{thm}
\begin{proof}
Assume that $\alpha$ starts a weak substantial gap $[\alpha,\beta]$ where $\beta<\omega_{\omega}^{CK,cc(\alpha)}$, so that $\beta<\omega_{i}^{CK,cc(\alpha)}$ for some minimal $i\in\omega$.
By definition, $\alpha$ is an index, so that $cc(\alpha)\in L_{\alpha+1}$. Passing to the $<_{L}$-smallest code not in $L_{\alpha}$ when necessary, we assume without loss of generality
that $cc(\alpha)\notin L_{\alpha}$. We now want to argue that $cc(\alpha)\in RECOG$, which will be a contradiction to the assumption that $\alpha$ starts a gap.
From $cc(\alpha)$, one can compute $cc(L_{\omega_{i}^{CK,cc(\alpha)}})$ by Theorem \ref{relITRM}. Let $P^{\prime}$ be an $ITRM$-program computing $cc(L_{\omega_{i}^{CK,cc(\alpha)}})$ in the oracle $cc(\alpha)$.
Since $i\in\omega$ is a fixed natural number, we can use $i$ together with $P^{\prime}$ to determine, for an arbitrary oracle $x$, whether $P^{\prime x}$ is a $<_{L}$-minimal code for $L_{\omega_{i}^{CK,x}}$. 
We can hence also compute the $<_{L}$-minimal code for $L_{\beta+1}$ in the oracle $cc(\alpha)$, using program $P$, say. 
By our assumption that $\beta$ ends the gap, we must have $RECOG\cap(L_{\beta+1}-L_{\beta})\neq\emptyset$; say $r\in RECOG\cap(L_{\beta+1}-L_{\beta})$, and let $Q$ be a program for recognizing $r$.
Now, given $x$ in the oracle, we can determine whether $P^{x}$ computes the minimal code for an $L$-level containing a real $z$ such that $Q^{z}\downarrow=1$.
(This can be achieved by searching through the coded structure; since $r$ is recognized by $Q$, the calculation  $Q^{z}$ will terminate for all reals $z$ from the coded structure.)
If this is not the case, then $x\neq cc(\alpha)$. Otherwise, $P^{x}$ has computed $cc(L_{\beta+1})$. In $cc(L_{\beta+1})$, the real $cc(\alpha)$ is represented by some fixed natural number $k\in\omega$ (which can hence be given to our program).
We can now simply test whether $x$ is the real coded by $k$ in $cc(L_{\beta+1})$ by bitwise comparison. This allows us to recognize $cc(\alpha)\in L_{\alpha+1}-L_{\alpha}$, which contradicts the assumption that $\alpha$ starts a gap.
\end{proof}


\begin{thm}
Let $\alpha$ start a weak substantial gap. Then $\alpha$ is admissibly $\Sigma_{1}$-indescribable.
\end{thm}
\begin{proof}
Assume for a contradiction that $\alpha$ is $\Sigma_{1}$-indescribable and starts a weak substantial gap. Then $\alpha$ is an Index, so that $cc(\alpha)\in L_{\alpha+1}-L_{\alpha}$ (assuming without loss of generality
that $cc(\alpha)\in L_{\alpha+1}$, passing to the $<_{L}$-minimal code of $\alpha$ not in $L_{\alpha}$ if necessary). 
Now, if $\alpha$ was admissibly $\Sigma_{1}$-describable, we could compute from $cc(\alpha)$ the $<_L$-minimal code of the first $L_{\beta}$ containing a witness for some $\Sigma_{1}$-statement $\phi$ which characterizes $\alpha$. 
Let $P$ be a program that achieves this. By the usual procedure, we can check for an arbitrary oracle $x$ whether $P^{x}$ computes a minimal code of a minimal $L$-level containing such a witness.
Now we must have $cc(\alpha)\in L_{\beta}$, so that $cc(\alpha)$ is represented in $cc(L_{\beta})$ by some fixed natural number $k$. To determine whether $x=cc(\alpha)$, it hence only remains to check
whether $x$ is equal to the number represented by $k$ in the structure coded by the real computed by $P^{x}$, which is also possible. So $cc(\alpha)\in L_{\alpha+1}-L_{\alpha}$ is recognizable,
contradicting the assumption that $\alpha$ starts a gap.
\end{proof}

By the same argument, we get:

\begin{corollary}
Let $[\alpha,\beta]$ be a strong substantial gap, and let $\gamma\in[\alpha,\beta]$. Then $\gamma$ is admissibly $\Sigma_{1}$-indescribable. 
\end{corollary}
\begin{proof}
 This follows by the same argument as above, since $\gamma$, being an element of a strong substantial gap, must be an index, which is the crucial property for this argument.
\end{proof}


\subsection{Antigaps}

\begin{defini}
$[\alpha,\beta]$ is a $\delta$-antigap if $\alpha+\delta\leq\beta$, the set of indices is unbounded below $\beta$ and, for each index $\alpha<\gamma<\beta$, $L_{\gamma+1}-L_{\gamma}$ contains a recognizable real.
\end{defini}

We can now demonstrate that potentially recognizable reals continue being recognizable for quite a while after $L_{\omega_{\omega}^{CK}}$:

\begin{thm}
 All elements of $L_{\omega_{\omega2}^{CK}}\cap PRECOG$ are recognizable.
\end{thm}
\begin{proof}
 (Sketch) It suffices to show that, given $x\in L_{\omega_{\omega2}^{CK}}$, the index $\gamma$ where $x$ appears has a recognizable $<_{L}$-minimal code. This can be seen as follows:
Let $i\in\omega$ be minimal such that $x\in L_{\omega_{\omega+i+1}^{CK}}$. Given $x$, one can compute (by $P$, say) the $<_{L}$-minimal code $c$ of $\omega_{\omega+i}^{CK}$. Since $i$ can be given
to the program explicitely, it is possible to determine for a given $y$ whether $P^{y}$ computes $c$. By definition of $\omega_{\omega+i+1}^{CK}$, there is a Turing program $T$ that computes the minimal code $c_{\gamma}$ for
$\gamma$ from $c$. $T$ can also be explicitely given to our program. Since $\gamma$ is an index, $c_{\gamma}\in (L_{\gamma+1}-L_{\gamma})$. Via the procedure just described (find $c$, then compute $T^{c}$), $c_{\gamma}$ is
recognizable. By the theorem above, every potentially recognizable real generated over $L_{\gamma}$ is hence recognizable. Hence $x\in RECOG$. 
\end{proof}

By the same reasoning, one can see that the first strong gap also appears above the first limit of limit of admissibles, the first limit of limits of limits of admissibles etc. The first gap corresponds to the first admissibly $\Sigma_{1}$-indescribable 
ordinal by Theorem \ref{char}, which is fairly high.\\

\section{Acknowledgments}

We are indebted to Philipp Schlicht for many helpful discussions on forcing over $KP$, sketching a proof of Lemma \ref{relbar} a crucial hint for the proof of Theorem \ref{wsigma1}
and suggesting several very helpful references.
We also thank Philipp Welch for suggesting the use of Cohen-forcing over $L_{\omega_{\omega}^{CK}}$ as a method for obtaining unrecognizables low down in the constructible hierarchy.

\end{document}